\documentclass{amsart}
\usepackage{amssymb,color,enumerate}
\usepackage{latexsym,url}
\usepackage{graphicx}
\usepackage{pdfsync,url}
\usepackage{amsmath}
\usepackage{amsfonts}
\allowdisplaybreaks

\usepackage[square,comma,sort&compress]{natbib}
\setcitestyle{numbers}
\usepackage{mathrsfs} 

\theoremstyle{plain}
 \newtheorem{thm}{Theorem}[section]
 \newtheorem{prop}{Proposition}[section]

\theoremstyle{definition}
 
 \newtheorem{rem}{Remark}[section]
 
\numberwithin{equation}{section}
\renewcommand{\leq}{\leqslant}
\renewcommand{\geq}{\geqslant}

\newcommand{\vague}{\stackrel{\lower0.2ex\hbox{$\scriptscriptstyle
                    \it{v} $}}{\rightarrow}}
\newcommand{\weak}{\stackrel{\lower0.2ex\hbox{$\scriptscriptstyle
                    \it{w} $}}{\rightarrow}}
\newcommand{\what}{\stackrel{\lower0.2ex\hbox{$\scriptscriptstyle
                    \it{\hat{w}} $}}{\rightarrow}}
\newcommand{\eqdis}{\stackrel{\lower0.2ex\hbox{$\scriptscriptstyle
                    \mathrm{d}$}}{=}}
\newcommand{\distr}{\stackrel{\lower0.2ex\hbox{$\scriptscriptstyle
                    \it{d} $}}{\rightarrow}}

\def\BX{\boldsymbol X}

\def\BC{\boldsymbol C}

\def\bb{\boldsymbol b}

\newcommand{\bbr}{\mathbb{R}}
\newcommand{\bbn}{\mathbb{N}}
\newcommand{\cF}{\mathcal{F}}
\newcommand{\one}{{\bf 1}}

\title[Count Normality]{Asymptotic Normality of Degree Counts in a
  Preferential Attachment Model}



\author[Resnick ]{Sidney I. Resnick}
\address{Sidney I. Resnick\\School of ORIE, Cornell University,
Ithaca, NY 14853} \email{sir1@cornell.edu}

\author[Samorodnitsky ]{Gennady Samorodnitsky}
\address{Gennady Samorodnitsky\\School of ORIE, Cornell University,
Ithaca, NY 14853} \email{gs18@cornell.edu}

\thanks{S. Resnick and G. Samorodnitsky were supported by Army MURI grant
  W911NF-12-1-0385 to Cornell University.} 

\begin{document}

\setcounter{page}{1}
\thispagestyle{empty}

\begin{abstract}
Preferential attachment is a widely adopted paradigm for understanding
the dynamics of 
 social networks.  Formal statistical inference,
for instance GLM techniques, and model
verification methods will require knowing test statistics are asymptotically
normal even though node or count based 
network data is nothing like  classical data from
independently replicated experiments. We therefore study asymptotic
normality of degree counts for a sequence of growing simple undirected
preferential attachment graphs. The methods of proof rely on
identifying martingales and then exploiting the martingale central
limit theorems.
\end{abstract}
\bibliographystyle{plainnat}
\maketitle

\section{Introduction}
Preferential attachment is a widely adopted model for understanding
social network growth. The assumption posits that nodes with a large
number of existing connections are more likely to attract connections from
new nodes joining the network. This paradigm is one of the
justifications for perceived power law behavior.

Statistical analyses of social networks is complicated by the fact
that node based data is nothing like classical iid data obtained from
repeated sampling. Efforts at statistical estimation and model
confirmation center around graphical based slope methods, regression
models for the center of the data and tail estimation
methodology. These statistical techniques are adaptations of classical
methods but  currently are largely without justification. Yet they
produce reasonable answers and plots.

We begin a program for justifying statistical methodology by examining
the asymptotic normality of counting variables for the number of nodes
of degree $k$ in a simple undirected preferential attachment model
 described in \cite[Chapter 8]{vanderHofstad:2014}. We
let $N_n(k)$ be the number of nodes of degree $k$ at the $n$th stage
of development of the network. It is known $N_n(k)/n  \to p_k$ and for
our model, the sequence $\{p_k\}$ can be exhibited explicitly. The
martingale central limit theorem allows us to prove asymptotic
normality for $\sqrt n (N_k(k)/n -p_k)$. This emphasizes consistency of the
empirical count percentages as estimates of $\{p_k\}$ and provides
confidence statements for the estimates.

We describe the undirected preferential attachment model in Section 
\ref{subsec:model} and give known mathematical results that we need in
Section \ref{sec:counts}.
Asymptotic normality is considered in the simplest case $k=1$ in
Section \ref{sec:counts}  and for the general case in Section \ref{subsec:k}.

\subsection{The model for preferential attachment.}\label{subsec:model}
We consider a simple growing undirected graph with preferential attachment that
is outlined in \cite[Chapter 8]{vanderHofstad:2014} or
\cite{durrett:2010}. The random graph 
at stage $n$ is $G_n=(V_n,E_n)$, the set of nodes or vertices is
$V_n:=\{1,\dots,n\}$ and the set of undirected edges is $E_n$, a
subset of $\{ \{i,j\}: i,j \in V_n\}$.
For $v\in V_n$, let $D_n(v)$ be the degree of $v$ at stage $n$; that
is, the number of edges incident to $v$. As a convenient and harmless
initialization, assume $V_1$ consists of a single node $1$ with a self-loop so that $D_1(1)=2.$

Conditional on knowing the graph $G_n$, at stage $n+1$ a new node $n+1$
appears and with a parameter $\delta > -1$, either
\begin{enumerate}
\item The new node $n+1$ attaches to $v\in V_n$ with probability
\begin{equation}\label{eq:attach1}
 \frac{D_n(v)+\delta}{n(2+\delta)+(1+\delta)},\end{equation}
or
\item $n+1$ attaches to itself 
with probability
\begin{equation}\label{eq:attach2}
  \frac{1+\delta}{n(2+\delta)+(1+\delta)}.\end{equation} 
\end{enumerate}
In the first case $D_{n+1}(n+1)=1$ and in the second
case $D_{n+1}(n+1)=2$. It is standard for this model that
$$\sum_{v\in V_n} D_n(v)=2n$$ and thus the attachment probabilities in
\eqref{eq:attach1}
and  \eqref{eq:attach2} 
add to 1.

\section{Martingale Central Limit theorem} \label{sec:MartCLT}

Martingale central limit theorems have been used for a long time. In
order to make the paper self-contained, we present the statements we
will need in this section. We start with a one-dimensional martingale
adaptation of the Lindeberg-Feller central limit theorem
(\cite[Chapter 8]{durrett:2010a} or 
\cite{hall:heyde:1980}). 
 
\begin{prop}\label{prop:mgCLT} Let $\{X_{n,m},\mathcal{F}_{n,m}, 1\leq
  m\leq n\} $ be a square integrable martingale difference array
  satisfying 
\begin{enumerate}
\item $V_{n,n}:=\sum_{m\leq n} E(X^2_{n,m}|\mathcal{F}_{n,m-1})
\stackrel{P}{\to} \sigma^2$ as $n\to\infty$.
\item $\sum_{m\leq n} E(X^2_{n,m}1_{[|X_{n,m} |>\epsilon]}|\mathcal{F}_{n,m-1} )
  \stackrel{P}{\to} 0$ as $n\to\infty$ for all $\epsilon >0$.
\end{enumerate}
Then as $n\to\infty$,
\begin{equation}\label{eq:mgCLT}
\sum_{m=1}^n X_{n,m} \Rightarrow N(0,\sigma^2). \end{equation}
\end{prop}

Using the Cram\'er-Wold device, it is not hard to extend the statement
of Proposition \ref{prop:mgCLT} to the multivariate case.

\begin{prop}\label{prop:mgCLTmult} 
Let $\{\BX_{n,m},\mathcal{F}_{n,m},
  1\leq   m\leq n\} $, $\BX_{n,m}=\bigl( X_{n,m,1},\ldots,
  X_{n,m,d}\bigr)^T$,  be a $d$-dimensional square-integrable martingale
  difference   array. Consider the $d\times d$ nonnegative definite
  random matrices 
$$
G_{n,m} = \Bigl( E\bigl(X_{n,m,i} X_{n,m,j}
\big|\mathcal{F}_{n,m-1}\bigr), \, i,j=1,\ldots, d\Bigr), \ \ 
V_n = \sum_{m=1}^n G_{n,m},
$$
and suppose $( A_n)$ is a sequence of $l\times d$ matrices with a bounded
supremum norm.  Assume that 
\begin{enumerate}
\item $A_n V_n A_n^T\stackrel{P}{\to} \Sigma$ as $n\to\infty$ for some
 nonrandom  (automatically nonnegatively definite) matrix $\Sigma$.  
\item $\sum_{m\leq n} E(X^2_{n,m,i}1_{[|X_{n,m,i}| >\epsilon]}|\mathcal{F}_{n,m-1} )
  \stackrel{P}{\to} 0$ as $n\to\infty$ for all $i=1,\ldots, d$ and
  $\epsilon >0$. 
\end{enumerate}
Then in $\mathbb{R}^l$, 
\begin{equation}\label{eq:mgCLTmult}
\sum_{m=1}^n A_n\BX_{n,m} \Rightarrow \BX, \quad (n\to\infty),
\end{equation}
a centered $l$-dimensional Gaussian vector with covariance matrix
$\Sigma$. 
\end{prop}
\begin{proof}
By the Cram\'er-Wold device, it is enough to prove that for any vector
$\bb\in \bbr^l$, 
$$
\sum_{m=1}^n \bb^T A_n\BX_{n,m} \Rightarrow
N\bigl(0, \bb^T\Sigma\bb\bigr)\,. 
$$
Since $\{\bb^T A_n\BX_{n,m},\mathcal{F}_{n,m}, 1\leq
  m\leq n\} $ is a one-dimensional square integrable martingale
  difference array, we only need to check the conditions of
  Proposition \ref{prop:mgCLT}. Since
$$
\sum_{m\leq n} E\Bigl( \bigl( \bb^T A_n\BX_{n,m}\bigr)^2
 \Big|\mathcal{F}_{n,m-1}\Bigr)
= \bb^T A_n V_n A_n^T \bb  \stackrel{P}{\to} \bb^T\Sigma\bb
$$
as $n\to\infty$, the first condition of that proposition holds. For
the second condition, notice that for every $\epsilon >0$ we can write 
$$
 \bigl| \bb^T A_n\BX_{n,m}\bigr| 
\leq d\max_{i=1,\ldots, d} \big| X_{n,m,i} \sum_{j=1}^l b_j
A_n(j,i)\big| \leq \theta \max_{i=1,\ldots, d} \big|
X_{n,m,i}\bigr|\,, 
$$
where $\theta$ is a finite positive constant depending on the vector
$\bb$ and an upper bound on the supremum norm of the sequence
$(A_n)$. Therefore,
$$
\sum_{m\leq n} E\Bigl( \bigl( \bb^T A_n\BX_{n,m}\bigr)^2 1_{\bigl|
  \bb^T A_n\BX_{n,m}\bigr|>\epsilon}  \Big|\mathcal{F}_{n,m-1}\Bigr) 
$$
$$
\leq \theta^2 \sum_{m\leq n} E\Bigl( \max_{i=1,\ldots, d} \big|
X_{n,m,i}\bigr|^2 1_{\max_{i=1,\ldots, d} |
X_{n,m,i}|>\epsilon/\theta}\Big|\mathcal{F}_{n,m-1}\Bigr) 
$$
$$
\leq \theta^2 \sum_{i=1}^d \sum_{m\leq n} E\bigl( X_{n,m,i}^2
1_{| X_{n,m,i}|>\epsilon/\theta}\big|\mathcal{F}_{n,m-1}\bigr)  
\to 0
$$
by the second assumption of the present proposition. This verifies the
assumptions of Proposition \ref{prop:mgCLT} and, hence, completes the
argument. 
\end{proof}

\section{Asymptotic normality of degree counts}\label{sec:counts}

For $k\geq 1$, let $N_n(k) $ be the number of nodes in $G_n$ with
degree $k$:
$$N_n(k)=\sum_{v\in V_n} 1_{[D_n(v)=k]}.$$
Using concentration inequalities and martingale methods it is shown,
for instance in \cite[Chapter 8]{vanderHofstad:2014} that there is a
probability mass function $\{p_k, k\geq 1\}$ such that almost surely
as $n\to\infty$,
\begin{equation} \label{e:freq.conv}
\frac{N_n(k)}{n} \to p_k,\quad k\geq 1\,,
\end{equation}
and
\begin{equation} \label{e:p.k}
p_k = \Bigl( \frac{(2+\delta)\Gamma(3+2\delta)    }{\Gamma(1+\delta)    }
\Bigr)
\frac{\Gamma(k+\delta)  }{\Gamma (k+3+2\delta)} =c(\delta) 
\frac{\Gamma(k+\delta)  }{\Gamma (k+3+2\delta)}.
\end{equation}
Of course $\Gamma (\cdot)$ is the gamma function. In particular, for
$k=1 $
$$p_1=\frac{2+\delta}{3+2\delta},$$
and as $k\to\infty$, power law behavior is
$$p_k \sim c(\delta) k^{-3-\delta}.$$

We consider the asymptotic normality of 
${N_n(k)}/{n} - p_k$ for $ k\geq 1$, and we start with the simplest
case of $k=1$.

When proving the asymptotic normality for the number of nodes of
degree 1, we will use the abbreviations
$$\nu_n=E(N_n(1)),\quad \gamma=\frac{1+\delta}{2+\delta},\quad
w_n=\frac{n}{n+\gamma} .$$
Note $N_1(1)=\nu_1=0$ since the initial node has a self-loop so
$D_1(1)=2$.

Let $\mathcal{F}_n$ be the information from observing the growth of
the network up through the $n$th stage. Then
\begin{align} \label{e:dyn.1}
E(N_{n+1}(1)|\mathcal{F}_n)&=N_n(1) +
                             E\bigl((N_{n+1}(1)-N_n(1))|\mathcal{F}_n\bigr)
  \notag \\
&=N_n(1) +
                             E\bigl(1_{[n+1 \text{ links to }v\in V_n;
  D_n(v)>1]}|\mathcal{F}_n\bigr) \notag \\
&=N_n(1) +1-P[n+1 \text{ links to itself }|\mathcal{F}_n ] \notag \\
{}&\qquad -P[ n+1 \text{ links to }v\in V_n; D_n(v)=1 | \mathcal{F}_n
    ] \notag \\
\intertext{and using \eqref{eq:attach2}  and the fact that for the
  last term $D_n(v)=1$, we get}
&=N_n(1) +1-   \frac{1+\delta}{n(2+\delta)+(1+\delta)} 
-\frac{1+\delta}{n(2+\delta)+(1+\delta)} N_n(1)\\
&=N_n(1)(1-\frac{\gamma}{n+\gamma})  +(1-\frac{\gamma}{n+\gamma})
  \notag \\
&=w_n N_n(1)+w_n. \notag
\end{align}
We conclude
\begin{align}
E(N_{n+1}(1)|\mathcal{F}_n)&=w_n N_n(1) +w_n , \label{eq:conditVer}\\
\nu_{n+1}&= w_n \nu_n + w_n. \label{eq:expVer}
\end{align}

We claim
\begin{equation}\label{eq:MG}
M_{n+1}:=\frac{N_{n+1}(1) -\nu_{n+1}}{\prod_{j=1}^n w_j}
=\frac{N_{n+1}(1) -\sum_{l=1}^n \prod_{j=l}^n w_j}
{\prod_{j=1}^n w_j} , \quad n \geq 1
\end{equation}
is a martingale. This is  verified using \eqref{eq:conditVer}
and \eqref{eq:expVer}.

Consider now the martingale difference
\begin{align}
d_{n+1}:=& M_{n+1}-M_n=\frac{N_{n+1}(1) -\nu_{n+1}}{\prod_{j=1}^n w_j}
-\frac{N_{n}(1) -\nu_{n}}{\prod_{j=1}^{n-1} w_j}\nonumber\\
=&\frac{1}{\prod_{j=1}^n w_j} \Bigl( N_{n+1}(1)-\nu_{n+1}-(w_nN_n(1)-w_n
  \nu_n)\Bigr)\nonumber\\
=&\frac{1}{\prod_{j=1}^n w_j} \Bigl(
   N_{n+1}(1)-N_n(1)+N_n(1)(1-w_n)-w_n\Bigr). \label{eq:d}
\end{align}
As above, 
$$N_{n+1}(1)-N_n(1)=:\Delta_{n+1} =1_{[n+1 \text{ links with }v\in
  V_n;D_n(v)>1]},$$
and $\Delta_{n+1}^2=\Delta_{n+1}$. Therefore,
\begin{align*}
d^2_{n+1} =&\frac{1}{(\prod_{j=1}^n w_j)^2} \Bigl(
\Delta_{n+1} +2\Delta_{n+1}\bigl(N_n(1)(1-w_n)-w_n\bigr)\\
&\qquad+\bigl(N_n(1)(1-w_n)-w_n
\bigr)^2
\Bigr)\\
=&\frac{1}{(\prod_{j=1}^n w_j)^2} \Bigl(
\Delta_{n+1} \{1+2\bigl((1-w_n)N_n(1)-w_n\bigr)\} \\
&\qquad +\bigl( (1-w_n)N_n(1)-w_n\bigr)^2
\Bigr).
\end{align*}
Since
$$E(\Delta_{n+1}|\mathcal{F}_n)=
E(N_{n+1}(1)-N_n(1) |\mathcal{F}_n)=1-\frac{\gamma}{n+\gamma}
(1+N_n(1)),$$
we have
\begin{align}
E(d^2_{n+1} |\mathcal{F}_n)=&
\frac{1}{(\prod_{j=1}^n w_j)^2} \Bigl(
\{1+2\bigl((1-w_n)N_n(1)-w_n\bigr)\}
                              \bigl(1-\frac{\gamma}{n+\gamma}(1+N_n(1)\bigr) 
\nonumber \\&\quad+[(1-w_n)N_n(1)-w_n]^2
\Bigr) \nonumber\\
=&: \frac{1}{(\prod_{j=1}^n w_j)^2}  \chi_n.\label{eq:conditVar}
\end{align}

We need the following observations:
\begin{enumerate}
\item[(a)] $w_n=n/(n+\gamma) \to 1, $ as $n\to\infty$.
\item[(b)] $1-w_n=\gamma/(n+\gamma)\sim \gamma/n\to 0.$
\item[(c)] Owing to the usual recursion on the gamma function,
$$\frac{\Gamma(n+1+\gamma)}{\Gamma(1+\gamma)}=
\prod_{j=0}^{n-1}(n+\gamma-j)=\prod_{l=1}^n(\gamma+l).$$
\item[(d)] Therefore we have
\begin{align}
\prod_{j=1}^n w_j =&\prod_{j=1}^n \Bigl(\frac{j}{j+\gamma}\Bigr)
=\frac{\Gamma (n+1)}{\Gamma (n+1+\gamma)} \Gamma(1+\gamma)\nonumber\\
\sim & \Gamma(1+\gamma) n^{-\gamma} =\Gamma (\frac{1}{p_1})
       n^{-\gamma}\label{eq:w} 
\end{align}
by Stirling's formula.
\item[(e)] We have
$$\lim_{n\to \infty} \frac{N_n(1)}{n}=\lim_{n\to\infty} \frac{\nu_n}{n}=p_1=\frac{2+\delta}{3+2\delta}.$$
\end{enumerate}

This allows us to evaluate $\lim_{n\to\infty} \chi_n$ in
\eqref{eq:conditVar} as
\begin{align}
c_0(\delta):=&\Bigl( 1+2(\gamma p_1 -1)][1-\gamma p_1]\Bigr) + (\gamma
             p_1 -1)^2 
=\gamma p_1(1-\gamma p_1)\nonumber\\ =&\Bigl(\frac{1+\delta}{3+2\delta}
\Bigr) \Bigl( \frac{2+\delta}{3+2\delta} \Bigr) =
\frac{(1+\delta)(2+\delta)}{(3+2\delta)^2}. \label{eq:c}
\end{align}
Therefore  from \eqref{eq:w}, \eqref{eq:conditVar} and \eqref{eq:c}
\begin{equation}\label{eq:Ed2}
E(d^2_{n+1} |\mathcal{F}_n) \sim 
\frac{c_0(\delta)}{(\Gamma(1+\gamma)n^{-\gamma})^2}
=:c_1(\delta) n^{2\gamma}
\end{equation}
with
$$c_1(\delta)=c_0(\delta)/(\Gamma(1+\gamma))^2.$$
Keeping in mind that the asymptotic equivalence in \eqref{eq:Ed2}
holds a.s, we apply Karamata's theorem on integration (eg, \cite[page
25]{resnick:2007} or \cite{bingham:goldie:teugels:1989}), to obtain 
\begin{equation}\label{eq:lim}
\sum_{j=1}^n E(d_{j+1}^2|\mathcal{F}_j) \sim
\frac{c_1(\delta)}{2\gamma+1} n^{2\gamma +1}=:\sigma^2 (\delta) n^{2\gamma +1}
,\quad (n\to\infty),
\end{equation}
with 
\begin{equation}\label{eq:sigma2}
\sigma^2(\delta)= \frac{(1+\delta)(2+\delta)}
{(3+2\delta)^2 (2\gamma+1) (\Gamma ( 1+\gamma))^2
}.
\end{equation}

Now set
$$d_{n,m}= \frac{d_m}{\sigma (\delta) n^{\gamma +1/2}}, \quad 1 \leq m
  \leq n,
$$
and the first condition of Proposition \ref{prop:mgCLT} is satisfied.
For the second condition of Proposition \ref{prop:mgCLT} observe that
from \eqref{eq:d}, we get by ignoring constants and remembering
$|\Delta_{m+1}|\leq 1$, that, as $m\to\infty$, 
\begin{align*}
[|d_{n,m+1} |>\epsilon]=&[  |d_{m+1} |>\epsilon n^{\gamma +1/2}]\\
\subset &
[\frac{1}{\prod_{j=1}^{m} w_j }| [1+N_m(1) (1-w_m) -w_m |>\epsilon
          n^{\gamma + 1/2}]\\
\subset & [cm^\gamma  |(const)| >\epsilon n^{\gamma +1/2}] \\
\subset & [cn^\gamma  |(const)| >\epsilon n^{\gamma +1/2}]. 
\end{align*}
So with probability converging to 1 as $n\to\infty$, all the indicator
functions  $1_{
[|d_{n,m} |>\epsilon]}$ vanish, and this verifies the
  second condition for the central limit theorem.

We conclude from Proposition \ref{prop:mgCLT} that \eqref{eq:mgCLT} holds.
Unpacking \eqref{eq:mgCLT}, we find
\begin{align*}
\sum_{m=1}^n d_{n,m}=& \sum_{m=1}^n \frac{d_m}{\sigma
  (\delta)n^{\gamma+1/2} }
=\frac{1}{\sigma( \delta) n^{\gamma+1/2}} \Bigl(\frac{N_n
  (1)-\nu_n}{\prod_{j=1}^{n-1} w_j } \Bigr ).
\end{align*}
Use \eqref{eq:w} to get
$$\sqrt n  \Bigl(\frac{N_n(1)}{n} -\frac{\nu_n}{n}
\Bigr)=
\frac{N_n-\nu_n}{\sqrt n} \Rightarrow N(0,\sigma_1^2(\delta)),$$
where $\sigma_1(\delta)=\Gamma (1+\gamma) \sigma(\delta).$
Since there exists $K$ such that 
$$\bigvee_{n=1}^\infty  |\nu_n -np_1| \leq K,$$ (eg, \cite[Section
8.5]{vanderHofstad:2014}),
we may conclude the
following.
\begin{prop}\label{prop:clt(k=1)} The number of nodes at stage $n$
  with degree 1 is asymptotically normal,
$$\sqrt n \Bigl(\frac{N_n(1)}{n} -p_1 \Bigr) \Rightarrow N(0,
\sigma^2_1(\delta)),$$ 
where  
$$\sigma_1^2(\delta)=\frac{(1+\delta)(2+\delta)^2}{(3+2\delta)^2 (4+3\delta)}.$$
\end{prop}

\section{Normality for the number of nodes of degree $k$,
  $k>1$.}\label{subsec:k}

The results of Section \ref{sec:counts} show how the martingale
central limit theorem can be used to prove that the deviation of the
fraction of the nodes in the $n$th graph, that have degree 1, from their limiting
fraction, is of the order $n^{-1/2}$ and, when normalized by that
quantity, has a limiting centered normal distribution; this is the
content of Proposition \ref{prop:clt(k=1)}. It turns out that this
distributional result is valid for all node degrees
simultaneously. More specifically, a central limit theorem in
$\bbr^\bbn$ holds, and the limit is a centered Gaussian process. 
\begin{thm} \label{t:main}
The proportion of nodes with given degrees satisfy the limiting
distributional relation
\begin{equation} \label{e:CLT.RN}
\left( \sqrt n \Bigl(\frac{N_n(k)}{n} -p_k \Bigr), \
  k=1,2,\ldots \right) \Rightarrow     \bigl( Z_k, \, k=1,2\ldots\bigr)
\end{equation}
in $\bbr^\bbn$, where $\bigl( Z_k, \, k=1,2\ldots\bigr)$ is a centered
Gaussian process with covariance function $R_Z$ given by
\eqref{final.cov}. 
\end{thm}
\begin{rem}
We precede the proof of the theorem with a number of comments.

First of all, weak convergence in $\bbr^\bbn$ is equivalent to
convergence of the finite-dimensional distributions; see
\cite{billingsley:1968}. 

It is clear also that, for every fixed $n$, the stochastic process in
the left hand side of \eqref{e:CLT.RN} will have at most $n$ random
elements; however, all elements in the limiting process are random.

Finally, the variance of $Z_1$ in the right hand side of
\eqref{e:CLT.RN} is given in Proposition \ref{prop:clt(k=1)}. 
\end{rem}
\begin{proof}[Proof of Theorem \ref{t:main}]
As in the one-dimensional case of Section \ref{sec:counts}, we will
use a martingale central limit theorem, so we start with constructing
a suitable martingale for each fixed node degree. For $k=1,2\ldots$
we denote $\nu_n(k)=E(N_n(k))$, $n\geq k$  (so that
$\nu_n=\nu_n(1)$). It follows from \eqref{e:freq.conv} and  bounded
convergence,
\begin{equation} \label{e:mean.conv}
\frac{\nu_n(k)}{n} \to p_k,\quad k\geq 1\,.
\end{equation}
Let 
\begin{equation} \label{e:a.nk}
a_n^{(k)} = \left[ \prod_{i=k}^{n-1} \left(
    1-\frac{k+\delta}{i(2+\delta)+1+\delta}\right) \right]^{-1}, \
n\geq k,
\end{equation}
so that
\begin{equation}\label{e:recura}
a_{n+1}^{(k)} =a_n^{(k)} \Bigl(
1-\frac{k+\delta}{n(2+\delta)+1+\delta}\Bigr)^{-1},
\end{equation}
and for future use, note that by the
Stirling formula, 
\begin{align} 
a_{n}^{(k)} =& \frac{\Gamma\bigl( n+(1+\delta)/(2+\delta)\bigr)
\Gamma\bigl( k+(1-k)/(2+\delta)\bigr)}
{\Gamma\bigl( k+(1+\delta)/(2+\delta)\bigr)
\Gamma\bigl( n+(1-k)/(2+\delta)\bigr)}\label{e:norm.a}\\
\sim& \frac{\Gamma\bigl( k+(1-k)/(2+\delta)\bigr)}
{\Gamma\bigl( k+(1+\delta)/(2+\delta)\bigr)}
n^{(k+\delta)/(2+\delta)} \qquad (n\to\infty)\notag
\end{align}
so that as a function of $n$, $a_n^{(k)} $ is  regularly varying
 with index $(k+\delta)/(2+\delta)$. Also define,
\begin{equation} \label{e:b.nk}
b_j^{(k)} =  \prod_{i=j}^{k-1}  \frac{i+\delta}{i-k} =
(-1)^{k-j}\frac{\Gamma(k+\delta)}{(k-j)!\, \Gamma(j+\delta)}, \ 1\leq
j\leq k.
\end{equation}
We use the usual conventions to set $a_k^{(k)}=b_k^{(k)}=1$, and set
\begin{equation} \label{e:M.n}
M_n^{(k)} = a_n^{(k)} \sum_{j=1}^k b_j^{(k)} \bigl( N_n(j)-
\nu_n(j)\bigr), \ n\geq k\,.
\end{equation}
The process $(M_n^{(1)} )$ coincides
with the process $(M_n)$ defined in \eqref{eq:MG} and we already proved
the latter process is a martingale with respect to the filtration
$(\mathcal{F}_n)$.  Now we check  that for each $k\geq 1$ the
process $\bigl(M_n^{(k)}, \, n\geq k \bigr)$ is a martingale with
respect to the same filtration.  

Recall the dynamics of the of the counting processes $(N_n(k))$: for
each fixed $n$,  there is a partition of the sample space 
into disjoint events $A_n(k)$, $k=0,1,\ldots$, $n\geq 1$, $n\geq k$,
with 
\begin{equation} \label{e:sel.prob}
P(A_n(k) |\mathcal{F}_n) = 
\left\{ \begin{array}{ll} 
\frac{k+\delta}{n(2+\delta)+(1+\delta)}
N_n(k), & k=1,\ldots, n\\ 
\frac{1+\delta}{n(2+\delta)+(1+\delta)}, &k=0\,.
\end{array}
\right.
\end{equation}
The event $A_n(k)$ is the event that a new node appears at stage $n+1$
and attaches to $v\in V_n$ with $D_n(v)=k$ while $A_n(0)$ is the event
that a new node attaches to itself.
In terms of these events, for
each $k\geq 3$ and $n\geq k$,
\begin{equation} \label{e:dynN.ge3}
N_{n+1}(k) = \left\{ \begin{array}{ll}
N_n(k)+1 & \text{on $A_n(k-1)$,} \\ 
N_n(k)-1 & \text{on $A_n(k)$,} \\ 
N_n(k) & \text{on $\cup_{j\notin \{ k-1,k\}} A_n(j),$}
\end{array} \right.
\end{equation}
for $k=2$ and $n\geq 2$,
\begin{equation} \label{e:dynN.2}
N_{n+1}(2) = \left\{ \begin{array}{ll}
N_n(2)+1 & \text{on $A_n(0)\cup A_n(1)$,} \\ 
N_n(2)-1 & \text{on $A_n(2)$,} \\ 
N_n(2) & \text{on $\cup_{j\geq 3} A_n(j),$}
\end{array} \right.
\end{equation}
while for $k=1$ and $n\geq 1$,
\begin{equation} \label{e:dynN.1}
N_{n+1}(1) = \left\{ \begin{array}{ll}
N_n(1)+1 & \text{on $\cup_{j\geq 2} A_n(j)$ ,} \\ 
N_n(1) & \text{on  $A_n(0)\cup A_n(1)$.}
\end{array} \right.
\end{equation}
In particular, for  $n\geq k$,
\begin{align} 
E(N_{n+1}(k)|\mathcal{F}_n) &= \left(1-
  \frac{k+\delta}{n(2+\delta)+(1+\delta)} \right) N_n(k) \label{e:dyn.kge3}\\
&\qquad +
\frac{k-1+\delta}{n(2+\delta)+(1+\delta)}  N_n(k-1), \ k\geq 3\,,
 \notag\\
E(N_{n+1}(2)|\mathcal{F}_n) &= \left(1-
  \frac{2+\delta}{n(2+\delta)+(1+\delta)} \right) N_n(2) \label{e:dyn.ke2}\\
&\qquad+
\frac{1+\delta}{n(2+\delta)+(1+\delta)}  \bigl( 1+N_n(1)\bigr)\,.
 \notag
\end{align}
Taking the expectation, we see that
\begin{align} 
\nu_{n+1}(k) &= \left(1-
  \frac{k+\delta}{n(2+\delta)+(1+\delta)} \right) \nu_n(k) \label{e:dyn.mkge3}\\
&\qquad +
\frac{k-1+\delta}{n(2+\delta)+(1+\delta)}  \nu_n(k-1), \ k\geq 3\,,  
 \notag \\
\nu_{n+1}(2) &= \left(1-
  \frac{2+\delta}{n(2+\delta)+(1+\delta)} \right) \nu_n(2)  \label{e:dyn.mke2}\\
&\qquad +
\frac{1+\delta}{n(2+\delta)+(1+\delta)}  \bigl( 1+\nu_n(1)\bigr)\,.
 \notag
\end{align}
The corresponding dynamics for $k=1$ is given in \eqref{e:dyn.1}. 
Therefore,  for $n\geq k$,
\begin{align*}
E\bigl( M_{n+1}^{(k)}\big| \mathcal{F}_n\bigr) 
&= a_{n+1}^{(k)} \sum_{j=1}^k b_j^{(k)} E\bigl[ \bigl( N_{n+1}(j)-
\nu_{n+1}(j)\bigr)\big| \mathcal{F}_n\bigr] \\
&= a_{n+1}^{(k)}  b_1^{(k)} \frac{n}{n+\gamma} \bigl( N_{n}(1)-
\nu_{n}(1)\bigr) \\
&\qquad + a_{n+1}^{(k)} \sum_{j=2}^k b_j^{(k)} \left[ \left(1-
  \frac{j+\delta}{n(2+\delta)+(1+\delta)} \right) \bigl(
N_n(j)-\nu_n(j)\bigr) \right. \\
& \left. 
\qquad + \frac{j-1+\delta}{n(2+\delta)+(1+\delta)} 
\bigl( N_n(j-1)-\nu_n(j-1)\bigr)\right] \\
 &= a_{n+1}^{(k)} \Biggl\{\sum_{j=1}^{k-1}\left[  b_j^{(k)} 
\left(1-  \frac{j+\delta}{n(2+\delta)+(1+\delta)} \right) \right. \\
&\left. 
\qquad + b_{j+1}^{(k)}\frac{j+\delta}{n(2+\delta)+(1+\delta)} \right] \bigl( N_{n}(j)-
\nu_{n}(j)\bigr) \\
&\qquad + b_k^{(k)} 
\left(1-  \frac{k+\delta}{n(2+\delta)+(1+\delta)} \right) \bigl( N_{n}(k)-
\nu_{n}(k)\bigr) \Biggr\}\,,
\end{align*}
and elementary calculations show that this is the same as the right
hand side of \eqref{e:M.n}. Therefore  for each $k$,  the process
$\bigl(M_n^{(k)}, 
\, n\geq k \bigr)$ is indeed
 a martingale with respect to the filtration
$(\mathcal{F}_n)$. 

For $k=1,2,\ldots$ define a
$k$-variate triangular array of martingale differences  by
\begin{equation} \label{e:mart.diff}
X_{n,m,j}=
\frac{M_{m}^{(j)}-M_{m-1}^{(j)}}{a_n^{(j)}n^{1/2}}, \ 
m=k+1,\ldots, n, \ j=1,\ldots, k\,, 
\end{equation}
for $n=k,k+1,\ldots$. In order to use the multivariate martingale
central limit theorem in Proposition \ref{prop:mgCLTmult}, we compute
 the asymptotic form of the quantities      
\begin{align} 
G_{n,m}(i,j) :=& E\bigl( X_{n,m,i}X_{n,m,j}\big| \cF_{m-1}\bigr) \label{e:mart.diff}\\
=&  \bigl( a_n^{(i)}a_n^{(j)}n\bigr)^{-1}
E\Bigl(  \bigl( M_{m}^{(i)}-M_{m-1}^{(i)}\bigr)\bigl(
M_{m}^{(j)}-M_{m-1}^{(j)} \bigr) \Big| \cF_{m-1}\Bigr)\,, \nonumber
\end{align}
$m=k+1,\ldots, n$, $i,j=1,\ldots, k$. 
By the martingale property,
\begin{align}  E\Bigl(  &\bigl( M_{m+1}^{(i)}-M_m^{(i)}\bigr)\bigl(
M_{m+1}^{(j)}-M_m^{(j)} \bigr) \Big| \cF_m\Bigr) \label{e:incr} \\
&= E\left[  \sum_{d=1}^i b_d^{(i)}\bigl(
    a_{m+1}^{(i)}N_{m+1}^{(d)} -     a_{m}^{(i)}N_{m}^{(d)}\bigr)           
\sum_{l=1}^j b_l^{(j)}\bigl(
    a_{m+1}^{(j)}N_{m+1}^{(l)} -     a_{m}^{(j)}N_{m}^{(l)}\bigr)\bigg| \cF_m\
  \right]\nonumber \\
& \qquad- \sum_{d=1}^i b_d^{(i)}\bigl(
    a_{m+1}^{(i)}\nu_{m+1}^{(d)} -     a_{m}^{(i)}\nu_{m}^{(d)}\bigr)           
\sum_{l=1}^j b_l^{(j)}\bigl(
    a_{m+1}^{(j)}\nu_{m+1}^{(l)} -     a_{m}^{(j)}\nu_{m}^{(l)}\bigr)\,.\nonumber
\end{align}
We begin by analyzing the behaviour of the deterministic term in the
right hand side of \eqref{e:incr}. We claim that for every $i\geq 1$,
\begin{equation} \label{e:incr.determ}
\lim_{n\to\infty} \frac{1}{a_{n}^{(i)}}\sum_{d=1}^i b_d^{(i)}\bigl(
    a_{n+1}^{(i)}\nu_{n+1}^{(d)} -     a_{n}^{(i)}\nu_{n}^{(d)}\bigr)
    = b_1^{(i)}\,,
\end{equation}
with $b_1^{(i)}$ given by \eqref{e:b.nk}. Indeed, for 
$d\geq 3$, by \eqref{e:a.nk} and \eqref{e:dyn.mkge3} we have 
\begin{align*}
a_{n+1}^{(i)}\nu_{n+1}^{(d)} -     a_{n}^{(i)}\nu_{n}^{(d)}
& = a_{n+1}^{(i)}\left( \nu_{n}^{(d)} \frac{i-d}{n(2+\delta)+1+\delta} 
+ \nu_{n}^{(d-1)} \frac{d-1+\delta}{n(2+\delta)+1+\delta} \right) \\
&\sim \frac{(i-d)p_d + (d-1+\delta)p_{d-1}}{2+\delta}\, a_n^{(i)}
\end{align*}
as $n\to\infty$ by \eqref{e:mean.conv}, the fact that
$a_{n+1}^{(i)}\sim a_{n}^{(i)}$  
and the same is true for $d=2$
by \eqref{e:dyn.mke2}. 
Similarly, using \eqref{e:dyn.1}, we obtain for $d=1$ that
$$
a_{n+1}^{(i)}\nu_{n+1}^{(1)} -     a_{n}^{(i)}\nu_{n}^{(1)}
\sim \left( \frac{i-1}{2+\delta}p_1 +1\right)a_n^{(i)} 
=a_n^{(i)}  +a_n^{(i)}\frac{i-1}{2+\delta} p_1,
$$
as $n\to\infty$. Therefore (with $p_0=0$),
\begin{align*}
&\lim_{n\to\infty} \frac{1}{a_{n}^{(i)}}\sum_{d=1}^i b_d^{(i)}\bigl(
    a_{n+1}^{(i)}\nu_{n+1}^{(d)} -     a_{n}^{(i)}\nu_{n}^{(d)}\bigr) \\
& = b_1^{(i)} + \frac{1}{2+\delta} \sum_{d=1}^i b_d^{(i)}\bigl[(i-d)p_d +
(d-1+\delta)p_{d-1}\bigr] = b_1^{(i)}\,,
\end{align*}
since
\begin{equation} \label{e:zero}
\sum_{d=1}^i b_d^{(i)}\bigl[ (i-d)p_d +
(d-1+\delta)p_{d-1}\bigr]
= \sum_{d=1}^{i-1} p_d \bigl[ (i-d) b_d^{(i)} + (d+\delta)
b_{d+1}^{(i)} \bigr] = 0\,.
\end{equation} 

Next, by \eqref{e:dynN.ge3}, \eqref{e:dynN.2} and \eqref{e:dynN.1},
$$
a_{n+1}^{(i)}N_{n+1}{(d)} -     a_{n}^{(i)}N_{n}{(d)} 
= a_{n+1}^{(i)} \left( N_{n}{(d)}
  \frac{i+\delta}{n(2+\delta)+1+\delta} + B_n(d)\right)\,,
$$
with 
\begin{align}
B_n(d) & = \left\{ \begin{array}{ll}
1 & \text{on $A_n(d-1)$,} \\ 
-1 & \text{on $A_n(d)$,} \\ 
0 & \text{on $\cup_{j\notin \{ d-1,d\}} A_n(j),$}
\end{array} \right. \ \ \text{for $d\geq 3$,} \notag \\
B_n(2)&  = \left\{ \begin{array}{ll}
1 & \text{on  $A_n(0)\cup A_n(1)$,} \\ 
-1 & \text{on $A_n(2)$,} \\ 
0 & \text{on $\cup_{j\geq 3} A_n(j),$}
\end{array} \right.  \label{e:Bn} \\
B_n(1)&  = \left\{ \begin{array}{ll}
1 & \text{on  $\cup_{j\geq 2} A_n(j)$ ,} \\ 
0 & \text{on $A_n(0)\cup A_n(1).$}
\end{array} \right. \notag
\end{align}
Therefore, as $n\to\infty$, 
\begin{align*}
 E&\left[  \frac{1}{a_n^{(i)}a_n^{(j)}}\sum_{d=1}^i b_d^{(i)}\bigl(
    a_{n+1}^{(i)}N_{n+1}{(d)} -     a_{n}^{(i)}N_{n}{(d)}\bigr)           
\sum_{l=1}^j b_l^{(j)}\bigl(
    a_{n+1}^{(j)}N_{n+1}{(l)} -     a_{n}^{(j)}N_{n}{(l)}\bigr)\bigg| \cF_n\
  \right]\\
&\sim  E\left[ \sum_{d=1}^i b_d^{(i)} \left( N_{n}{(d)}
  \frac{i+\delta}{n(2+\delta)+1+\delta} + B_n(d)\right) \times
\right.\\
&\left. 
\qquad \sum_{l=1}^j b_l^{(j)} \left( N_{n}{(l)}
  \frac{j+\delta}{n(2+\delta)+1+\delta} + B_n(l)\right) \bigg| \cF_n\
  \right]
\\
&= \sum_{d=1}^i b_d^{(i)}
(i+\delta)\frac{N_{n}{(d)}}{n(2+\delta)+1+\delta} 
\sum_{l=1}^j b_l^{(j)} (j+\delta)\frac{N_{n}{(l)}}{n(2+\delta)+1+\delta}
\\
&\qquad + \sum_{d=1}^i b_d^{(i)}
(i+\delta)\frac{N_{n}{(d)}}{n(2+\delta)+1+\delta} \sum_{l=1}^j
b_l^{(j)} E\bigl( B_n(l)\big| \cF_n\bigr)
\\
&\qquad 
+ \sum_{l=1}^j b_l^{(j)}
(j+\delta)\frac{N_{n}{(l)}}{n(2+\delta)+1+\delta}
 \sum_{d=1}^i b_d^{(i)} E\bigl( B_n(d)\big| \cF_n\bigr)\\
&\qquad+ \sum_{d=1}^i  \sum_{l=1}^j  b_d^{(i)} b_l^{(j)}
E\bigl( B_n(d)B_n(l)\big| \cF_n\bigr)
\\
&
:= S_{1,n}(i,j) + S_{2,n}(i,j) + S_{3,n}(i,j) +  S_{4,n}(i,j)\,.
\end{align*}
It follows by \eqref{e:freq.conv} that
$$
\sum_{d=1}^i b_d^{(i)}
(i+\delta)\frac{N_{n}{(d)}}{n(2+\delta)+1+\delta}
\to \frac{i+\delta}{2+\delta} \sum_{d=1}^i b_d^{(i)}p_d \ \
\text{a.s. as $n\to\infty$.}
$$

Next, by \eqref{e:Bn}, 
\begin{align}
 \sum_{d=1}^i & b_d^{(i)} E\bigl( B_n(d)\big| \cF_n\bigr)
= b_1^{(i)}  \left[ 1-\frac{1+\delta}{n(2+\delta)+1+\delta}\bigl(1+
  N_n{(1)}\bigr)\right] \notag \\
&+ \one_{i\geq 2} \, b_2^{(i)} \left[ \frac{1+\delta}{n(2+\delta)+1+\delta}\bigl(1+
  N_n{(1)}\bigr) - \frac{2+\delta}{n(2+\delta)+1+\delta}
  N_n{(2)}\right] \notag \\
&+ \one_{i\geq 3} \sum_{d=3}^i b_d^{(i)} \left[
  \frac{d-1+\delta}{n(2+\delta)+1+\delta} 
  N_n{(d-1)}- \frac{d+\delta}{n(2+\delta)+1+\delta}
  N_n{(d)}\right]  \notag \\
\to & b_1^{(i)}\left( 1- \frac{1+\delta}{2+\delta}p_1\right)
  + \one_{i\geq 2} \frac{1}{2+\delta}\sum_{d=2}^i b_d^{(i)} \bigl( (d-1+\delta)p_{d-1}
  - (d+\delta)p_d\bigr) \notag \\
=&  b_1^{(i)} +\frac{1}{2+\delta} \sum_{d=1}^i b_d^{(i)}
  \bigl( (d-1+\delta)p_{d-1}   - (d+\delta)p_d\bigr) \notag \\
=&  b_1^{(i)} - \frac{i+\delta}{2+\delta} \sum_{d=1}^i  b_d^{(i)}p_d \notag
\end{align}
a.s., where at the last step we used \eqref{e:zero}. We conclude that,
with probability 1, 
\begin{align*}
&S_{1,n}(i,j) \to \frac{(i+\delta)(j+\delta)}{(2+\delta)^2} 
\sum_{d=1}^i b_d^{(i)}p_d \sum_{l=1}^j b_l^{(j)} p_l\,,
\\
&S_{2,n}(i,j) \to \frac{i+\delta}{2+\delta} \sum_{d=1}^i  b_d^{(i)}p_d
\left( b_1^{(j)} - \frac{j+\delta}{2+\delta} \sum_{l=1}^j
  b_l^{(j)}p_l\right)\,, 
\\
&S_{3,n}(i,j) \to \frac{j+\delta}{2+\delta} \sum_{l=1}^ij b_l^{(j)}p_l
\left( b_1^{(i)} - \frac{i+\delta}{2+\delta} \sum_{d=1}^i
  b_d^{(i)}p_d\right)\,. 
\end{align*}
Finally, we consider the term $S_{4,n}(i,j)$. Note that, by
\eqref{e:Bn}, we have the following cases.
\begin{enumerate}
\item On $A_n(0) \cup A_n(1)$: 
$$
E\bigl( B_n(d)B_n(l)\big| \cF_n\bigr) = \one_{d=l=2}\,.
$$
\item On $A_n(m)$, $m\geq 2$,
$$
E\bigl( B_n(d)B_n(l)\big| \cF_n\bigr) = \left\{
\begin{array}{ll}
{}1 & \text{if $d,l\in \{ 1,m+1\}$ or $d=l=m$,} \\
-1 & \text{if $d=m$, $l\in \{ 1, m+1\}$ or $l=m$, $d\in \{ 1,m+1\}$.}
\end{array}
\right.
$$
\end{enumerate} 
Therefore, using the convention $b_d^{(i)}=0$ if $d>i$, we can write 
\begin{align*}
& S_{4,n}(i,j) = 
b_2^{(i)}b_2^{(j)}\frac{1+\delta}{n(2+\delta)+1+\delta}\bigl(1+ 
  N_n^{(1)}\bigr) \\
& + \sum_{m=2}^n \bigl( b_1^{(i)} -b_m^{(i)} +  b_{m+1}^{(i)}\bigr) \bigl(
  b_1^{(j)} -b_m^{(j)} +  b_{m+1}^{(j)}\bigr) 
\frac{m+\delta}{n(2+\delta)+1+\delta}   N_n^{(m)}  \\
&\to 
\sum_{m=1}^\infty  \frac{m+\delta}{2+\delta}\bigl( b_1^{(i)} -b_m^{(i)} +
  b_{m+1}^{(i)}\bigr) \bigl(   b_1^{(j)} -b_m^{(j)} +  b_{m+1}^{(j)}\bigr) p_m
\end{align*}
a.s. as $n\to\infty$. We conclude that, with probability 1, 
\begin{align} \label{e:limint.M}
& E \left[ \frac{1}{a_n^{(i)}a_n^{(j)}} \bigl( M_{n+1}^{(i)}-M_n^{(i)}\bigr)\bigl(
M_{n+1}^{(j)}-M_n^{(j)} \bigr) \Big| \cF_n\right] \to a(i,j) \\
& =: \sum_{m=1}^\infty  \frac{m+\delta}{2+\delta}\bigl( b_1^{(i)} -b_m^{(i)} +
  b_{m+1}^{(i)}\bigr) \bigl(   b_1^{(j)} -b_m^{(j)} +
  b_{m+1}^{(j)}\bigr) p_m \notag \\
&\qquad -\left( b_1^{(i)} -\frac{i+\delta}{2+\delta}\sum_{d=1}^i
  b_d^{(i)}p_d\right)
\left( b_1^{(j)} -\frac{j+\delta}{2+\delta}\sum_{l=1}^j
  b_l^{(j)}p_l\right). \notag
\end{align}

Next we   simplify the expression for $a(i,j)$, to make it clear
that $a(i,i)>0$ for all $i\geq 1$. Note that, by \eqref{e:p.k} and
\eqref{e:b.nk},  for each $i$, 
\begin{align*}
& \sum_{d=1}^{i-1}  b_d^{(i)}p_d
=  c(\delta)  (-1)^i \Gamma(i+\delta) \sum_{d=1}^{i-1} (-1)^d
\frac{1}{(i-d)!\, \Gamma(d+3+2\delta)} \\
& = c(\delta)  \frac{(-1)^i\Gamma(i+\delta)}{i+2+2\delta}
 \sum_{d=1}^{i-1} (-1)^d 
\left[ \frac{1}{(i-d)!\, \Gamma(d+2+2\delta)} + 
\frac{1}{(i-d-1)!\, \Gamma(d+3+2\delta)} \right] \\
& = c(\delta)  \frac{(-1)^i\Gamma(i+\delta)}{i+2+2\delta}
\left( -\frac{1}{(i-1)!\, \Gamma(3+2\delta)}
  +(-1)^{i-1}\frac{1}{\Gamma(i+2+2\delta)} \right)\,,
\end{align*}
where at the last step we used the telescoping property of the
sum. Elementary algebra now gives us for $i\geq 1$,
\begin{equation} \label{e:simplify.1}
b_1^{(i)} -\frac{i+\delta}{2+\delta}\sum_{d=1}^i
  b_d^{(i)}p_d = \frac{2+\delta}{\Gamma(1+\delta)}
  \frac{(-1)^{i-1}\Gamma(i+\delta)}{(i-1)!\, \Gamma(i+2+2\delta)}.
\end{equation}

Similarly,  for $i\geq 1$, using a telescopic
property, 
\begin{align*}
& \sum_{m=i+1}^\infty  \frac{m+\delta}{2+\delta}\bigl( b_1^{(i)} -b_m^{(i)} +
  b_{m+1}^{(i)}\bigr)^2 p_m 
=
\bigl( b_1^{(i)}\bigr)^2 \frac{c(\delta)}{2+\delta}
  \sum_{m=i+1}^\infty \frac{\Gamma(m+1+\delta)}{\Gamma(m+3+2\delta)} \\
& =\bigl( b_1^{(i)}\bigr)^2 \frac{c(\delta)}{2+\delta}
\sum_{m=i+1}^\infty 
\frac{1}{1+\delta} \left[
  \frac{\Gamma(m+1+\delta)}{\Gamma(m+2+2\delta)} -
  \frac{\Gamma(m+2+\delta)}{\Gamma(m+3+2\delta)} \right] \\
&= \bigl( b_1^{(i)}\bigr)^2 \frac{c(\delta)}{(1+\delta)(2+\delta)}
\frac{\Gamma(i+2+\delta)}{\Gamma(i+3+2\delta)} \\
& = \frac{\Gamma(3+2\delta)}{\Gamma(2+\delta)\bigl(
  \Gamma(1+\delta)\bigr)^2} \frac{\Gamma(i+2+\delta)\bigl(
  \Gamma(i+\delta)\bigr)^2}{ \Gamma(i+3+2\delta) \bigl(
  (i-1)!\bigr)^2}\,.
\end{align*}
Therefore, by \eqref{e:simplify.1},
\begin{align*}
&a(i,i)>  \sum_{m=i+1}^\infty  \frac{m+\delta}{2+\delta}\bigl( b_1^{(i)} -b_m^{(i)} +
  b_{m+1}^{(i)}\bigr)^2 p_m - \left( b_1^{(i)} -\frac{i+\delta}{2+\delta}\sum_{d=1}^i
  b_d^{(i)}p_d\right)^2 \\
&= \left( \frac{2+\delta}{\Gamma(1+\delta)}
  \frac{\Gamma(i+\delta)}{(i-1)!\, (i+2+2\delta)}\right)^2
\left[ \frac{\Gamma(3+2\delta)}{(2+\delta)\Gamma(3+\delta)}
\frac{(i+2+2\delta)\Gamma(i+2+\delta)}{\Gamma(i+2+2\delta)}-1\right].
\end{align*}
Note that the expression inside the bracket is greater than
$$
\frac{\Gamma(3+2\delta)}{(2+\delta)\Gamma(3+\delta)} 
\frac{ \Gamma(i+2+\delta)}{\Gamma(i+1+2\delta)}-1\,,
$$
and, since for $0<b<a$, the ratio $\Gamma(x+a)/\Gamma(x+b)$ is
increasing in $x\geq 0$, the above difference is, for $i\geq 2$, at
least
$$
\frac{\Gamma(3+2\delta)}{(2+\delta)\Gamma(3+\delta)} 
\frac{ \Gamma(2+2+\delta)}{\Gamma(2+1+2\delta)}-1=
\frac{3+\delta}{2+\delta}-1 >0\,.
$$
This shows that $a(i,i)>0$ for all $i\geq 2$. The fact that the same
is true for $i=1$ can be seen directly from \eqref{e:limint.M} (and
was also shown in Section \ref{sec:counts}). 

We know from \eqref{e:limint.M}  that
$$nG_{n,n}(i,j)\to a(i,j),\qquad (n\to\infty),$$
and from the definition \eqref{e:mart.diff} we have
$$G_{n,m}=\frac{a_m^{(i)}a_m^{(j)}}{na_n^{(i)}a_n^{(i)}} mG_{m,m}(i,j).$$
and from the regular variation property after \eqref{e:norm.a}, the
function
$$u(m):=a_m^{(i)}a_m^{(i)} mG_{m,m}(i,j)$$ is regularly varying with
index $(i+j+2\delta)/(2+\delta)$. Therefore, from Karamata's theorem
on integration of regularly varying functions
\begin{align*}
V_n(i,j)=&\sum_{m=k+1}^n G_{n,m}(i,j) =\frac{\sum_{m=k}^n u(m)}{n
           a_n^{(i)}a_n^{(j)} }\\
\sim & \frac{nu(n)}{(\frac{i+j+2\delta}{2+\delta} +1)
       na_n^{(i)}a_n^{(i)}}
\sim  a(i,j)\frac{2+\delta}{i+j+3\delta+2},
\end{align*}
for $i,j=1,\dots,k$, 
where $a(i,j)$ is defined in
\eqref{e:limint.M}. 
This verifies the first condition  the martingale central limit
theorem of Proposition \ref{prop:mgCLTmult} (with each $A_n$ being a
$k\times k$ identity matrix.) The second condition of
the theorem holds as well, since by \eqref{e:dynN.ge3}, the
differences are bounded, 
$$
\big|\bigl( N_n(j)-\nu_n(j)\bigr) - \bigl( N_{n-1}(j)-
\nu_{n-1}(j)\bigr)\big| \leq 2 \ \ \text{for all $j$,}
$$
hence, as in the one-dimensional case of Section \ref{sec:counts}, for
all $n$ large enough, the events $\bigl\{ X_{n,m,j} >\epsilon\bigr\}$
are empty for all $m\leq n$ and all $j$. We conclude that 
\begin{equation} \label{e:conv.Y}
\left( \frac{1}{n^{1/2}}\sum_{j=1}^k b_j^{(k)} \bigl( N_n(j)-
\nu_n(j)\bigr), \, k\geq 1\right) \Rightarrow \bigl( Y_k, \, k=1,2\ldots\bigr)
\end{equation}
in $\bbr^\bbn$, where $\bigl( Y_k, \, k=1,2\ldots\bigr)$ is a centered
Gaussian process with covariance function $R_Y$ given by
\begin{equation} \label{e:cov.Y}
R_Y(i,j) = \frac{2+\delta}{i+j+2+3\delta} \, a(i,j), \ i,j\geq 1\,.
\end{equation} 
We use this covariance function to define the $k\times k$ matrix
$$R_{Y,k}=(R_Y(i,j), 1 \leq i,j\leq k).$$

For a fixed $k=1,2\ldots$ the convergence in \eqref{e:conv.Y} means
that
$$
C_k\left( \frac{N_n(j)-\nu_n(j)}{n^{1/2}}, \, j=1,\ldots,k\right)^{
  T} \Rightarrow \bigl( Y_j, \, j=1,\ldots,k\bigr)\,,
$$
where $C_k$ is a $k\times k$ matrix with the entries
$$
c_{i,j} = \left\{ \begin{array}{ll}
b^{(i)}_j & j\leq i \\
0 & j>i\,.
\end{array}
\right.
$$
Using the easily checkable identity, valid for any real $r$, 
\begin{equation} \label{e:ident.b}
\sum_{m=j}^i r^{m-j}b_m^{(i)}b_j^{(m)}= b^{(i)}_j(1+r)^{i-j}, \ 1\leq
j\leq i\,,
\end{equation} 
we can check that the inverse of $C_k$, $D_k=C_k^{-1}$, has the
entries
\begin{equation} \label{e:inverse}
d_{i,j} = \left\{ \begin{array}{ll}
(-1)^{i-j}b^{(i)}_j & j\leq i \\
0 & j>i\,.
\end{array}
\right.
\end{equation}
We conclude that 
$$
\left( \frac{N_n(j)-\nu_n(j)}{n^{1/2}}, \, j=1,\ldots,k\right)
  \Rightarrow D_k\bigl( Y_j, \, j=1,\ldots,k\bigr)^T\,,
$$
and the covariance matrix of the limiting Gaussian vector is given by 
$$
\Sigma_k = D_kR_{Y,k}D_k^T\,,
$$
where $R_{Y,k}$ is the $k\times k$ matrix given after
\eqref{e:cov.Y}. 

In order to facilitate the calculation of the entries of the matrix
$\Sigma_k$, we write the matrix $R_{Y,k}$ in the form
$$
R_{Y,k} = \sum_{m=0}^\infty h_m\int_0^\infty (2+\delta)e^{-(2+3\delta)x}
R_{m,x}\, dx\,,
$$
where 
$$
h_m=\left\{ \begin{array}{ll}
-1 & m=0,\\
\frac{m+\delta}{2+\delta}p_m & m=1,2,\ldots\,,
\end{array}
\right.
$$
and the matrix $R_{m,x}$ is a $k\times k$ matrix  of the form
$$
R_{m,x} = \BC_{m,x}^T \BC_{m,x}\,.
$$
Here $\BC_{m,x}$ is a vector with the entries
$$
C_{m,x}(i) = \left\{ \begin{array}{ll}
\left( b_1^{(i)} -\frac{i+\delta}{2+\delta}\sum_{d=1}^i
  b_d^{(i)}p_d\right) e^{-ix}, \, i\geq 1    & m= 0\\
\bigl(b_1^{(i)} -b_m^{(i)} +  b_{m+1}^{(i)}\bigr) e^{-ix}, \, i\geq 1
                       & m\geq 1\,.
\end{array}
\right.
$$
Therefore,
$$
\Sigma_k = \sum_{m=0}^\infty h_m\int_0^\infty (2+\delta)e^{-(2+3\delta)x}
D_k \BC_{m,x}^T \BC_{m,x} D_k^T \, dx\,.
$$
 
Note that by \eqref{e:inverse} and \eqref{e:ident.b}, 
for any $m\geq 1$ and $i=1,2,\ldots$,
$$
\bigl( D_k \BC_{m,x}^T\bigr)(i) = \sum_{j=1}^i (-1)^{i-j}b^{(i)}_j
\bigl(b_1^{(j)} -b_m^{(j)} +  b_{m+1}^{(j)}\bigr) e^{-jx} 
$$
$$
= e^{-x} b_1^{(i)}\bigl(-1+e^{-x}\bigr)^{i-1}
- e^{-mx} b_m^{(i)}\bigl(-1+e^{-x}\bigr)^{i-m}
+ e^{-(m+1)x} b_{m+1}^{(i)}\bigl(-1+e^{-x}\bigr)^{i-m-1}\,.
$$
Therefore, for $m\geq 1$ and $i,j=1,2,\ldots$,
\begin{align*}
\bigl( D_k \BC_{m,x}^T\bigr)(i) \bigl( D_k \BC_{m,x}^T\bigr)(j)
& = b_1^{(i)}b_1^{(j)}e^{-2x} \bigl(-1+e^{-x}\bigr)^{i+j-2} \\
& - \bigl( b_1^{(i)}b_m^{(j)}+  b_m^{(i)}b_1^{(j)}\bigr) e^{-(m+1)x}
  \bigl(-1+e^{-x}\bigr)^{i+j-m-1} \\
& + \bigl( b_1^{(i)}b_{m+1}^{(j)}+  b_{m+1}^{(i)}b_1^{(j)}\bigr) e^{-(m+2)x}
  \bigl(-1+e^{-x}\bigr)^{i+j-m-2}\\
&+  b_m^{(i)}b_m^{(j)}e^{-2mx} \bigl(-1+e^{-x}\bigr)^{i+j-2m} \\
&- \bigl( b_m^{(i)}b_{m+1}^{(j)}+  b_{m+1}^{(i)}b_m^{(j)}\bigr) e^{-(2m+1)x}
  \bigl(-1+e^{-x}\bigr)^{i+j-2m-1} \\
&+ b_{m+1}^{(i)}b_{m+1}^{(j)}e^{-(2m+2)x}
  \bigl(-1+e^{-x}\bigr)^{i+j-2m-2} \\
& := \sum_{l=1}^6 \theta_{m,x}^{(l)}(i,j)\,.
\end{align*}
We have:
\begin{align*}
& \int_0^\infty (2+\delta)e^{-(2+3\delta)x} \theta_{m,x}^{(1)}(i,j) \,
dx \\
& = (-1)^{i+j}  (2+\delta) b_1^{(i)}b_1^{(j)}
\int_0^\infty  e^{-(4+3\delta)x} \bigl(1-e^{-x}\bigr)^{i+j-2}\, dx \\
&=  (-1)^{i+j}  (2+\delta) b_1^{(i)}b_1^{(j)}   B(4+3\delta,i+j-1) \\
&=  (-1)^{i+j}  (2+\delta) b_1^{(i)}b_1^{(j)}
  \frac{\Gamma(4+3\delta)(i+j-2)!}{\Gamma(i+j+3+3\delta)}\,.
\end{align*}
Similarly,
\begin{align*}
& 
\int_0^\infty (2+\delta)e^{-(2+3\delta)x} \theta_{m,x}^{(2)}(i,j) \,
dx  \\
& = (-1)^{i+j-m}  (2+\delta)\bigl( b_1^{(i)}b_m^{(j)}+
b_m^{(i)}b_1^{(j)}\bigr)
 \frac{\Gamma(m+3+3\delta)(i+j-m-1)!}{\Gamma(i+j+3+3\delta)}\,,
\end{align*}
\begin{align*}
& 
\int_0^\infty (2+\delta)e^{-(2+3\delta)x} \theta_{m,x}^{(3)}(i,j) \,
dx  \\
& = (-1)^{i+j-m}  (2+\delta)\bigl( b_1^{(i)}b_{m+1}^{(j)}+
b_{m+1}^{(i)}b_1^{(j)}\bigr)
 \frac{\Gamma(m+4+3\delta)(i+j-m-2)!}{\Gamma(i+j+3+3\delta)}\,,
\end{align*}
\begin{align*}
& 
\int_0^\infty (2+\delta)e^{-(2+3\delta)x} \theta_{m,x}^{(4)}(i,j) \,
dx  \\
& = (-1)^{i+j}  (2+\delta) b_m^{(i)}b_m^{(j)}
 \frac{\Gamma(2m+2+3\delta)(i+j-2m)!}{\Gamma(i+j+3+3\delta)}\,,
\end{align*}
\begin{align*}
& 
\int_0^\infty (2+\delta)e^{-(2+3\delta)x} \theta_{m,x}^{(5)}(i,j) \,
dx  \\
& = (-1)^{i+j}  (2+\delta) \bigl( b_m^{(i)}b_{m+1}^{(j)}+
  b_{m+1}^{(i)}b_m^{(j)}\bigr)  
 \frac{\Gamma(2m+3+3\delta)(i+j-2m-1)!}{\Gamma(i+j+3+3\delta)}\,,
\end{align*}
\begin{align*}
& 
\int_0^\infty (2+\delta)e^{-(2+3\delta)x} \theta_{m,x}^{(6)}(i,j) \,
dx  \\
& = (-1)^{i+j}  (2+\delta) b_{m+1}^{(i)}b_{m+1}^{(j)}
 \frac{\Gamma(2m+4+3\delta)(i+j-2m-2)!}{\Gamma(i+j+3+3\delta)}\,.
\end{align*}

Similarly, by \eqref{e:inverse} and \eqref{e:simplify.1}, for $i\geq
1$, 
$$
\bigl( D_k \BC_{0,x}^T\bigr)(i) =
\frac{(2+\delta)\Gamma(i+\delta)}{\Gamma(1+\delta)} 
 \sum_{l=1}^i (-1)^{l-1}e^{-lx}
 \frac{1}{(i-l)!(l-1)!\Gamma(l+2+2\delta)}\,.
$$
Therefore, for $i,j\geq 1$, 
$$
\int_0^\infty (2+\delta)e^{-(2+3\delta)x} \bigl( D_k \BC_{0,x}^T\bigr)(i)
\bigl( \BC_{0,x} D_k^T \bigr)(j)\, dx
$$
$$
=\frac{(2+\delta)^2\Gamma(i+\delta)\Gamma(j+\delta)}{\bigl(\Gamma(1+\delta)\bigr)^2} 
$$
$$
\sum_{l_1=1}^i \sum_{l_2=1}^j\frac{(-1)^{l_1+l_2}(l_1+l_2+2+3\delta)^{-1}}{(i-l_1)!(l_1-1)!
\Gamma(l_1+2+2\delta)
(j-l_2)!(l_2-1)!\Gamma(l_2+2+2\delta)}\,.
$$
We conclude that the covariance function of the limiting Gaussian
process $\bigl( Z_k, \, k=1,2\ldots\bigr)$ in \eqref{e:CLT.RN} is
given by 
\begin{align} \label{final.cov}
R_Z(i,j) &= \frac{(-1)^{i+j}}{\Gamma(i+j+3+3\delta)} \notag \\
&\sum_{m=1}^\infty 
 (m +\delta)p_m \bigl[ b_1^{(i)}b_1^{(j)} \Gamma(4+3\delta)(i+j-2)!
  \\
& + (-1)^m \bigl( b_1^{(i)}b_m^{(j)}+
b_m^{(i)}b_1^{(j)}\bigr) \Gamma(m+3+3\delta)(i+j-m-1)!  \notag \\
&+  (-1)^m \bigl( b_1^{(i)}b_{m+1}^{(j)}+
b_{m+1}^{(i)}b_1^{(j)}\bigr) \Gamma(m+4+3\delta)(i+j-m-2)!  \notag \\
&+ b_m^{(i)}b_m^{(j)} \Gamma(2m+2+3\delta)(i+j-2m)!  \notag \\
&+ \bigl( b_m^{(i)}b_{m+1}^{(j)}+
  b_{m+1}^{(i)}b_m^{(j)}\bigr)  \Gamma(2m+3+3\delta)(i+j-2m-1)! \notag \\
&+  b_{m+1}^{(i)}b_{m+1}^{(j)} \Gamma(2m+4+3\delta)(i+j-2m-2)!\bigr]\notag 
  \\
&-
  \frac{(2+\delta)^2\Gamma(i+\delta)\Gamma(j+\delta)}{\bigl(\Gamma(1+\delta)\bigr)^2}
 \notag  \\ 
& \sum_{l_1=1}^i \sum_{l_2=1}^j\frac{(-1)^{l_1+l_2}(l_1+l_2+2+3\delta)^{-1}}{(i-l_1)!(l_1-1)!
\Gamma(l_1+2+2\delta)
(j-l_2)!(l_2-1)!\Gamma(l_2+2+2\delta)}\,. \notag 
\end{align}
Once again, it is possible to show that $R_Z(i,i)>0$ for all $i\geq
1$. We omit the argument. 
\end{proof}

\section{Concluding remarks}
The model considered here is relatively simple and the calculations
are relatively complex and it remains to be seen which more realistic
models allow us to successfully conclude asymptotic normality of
degree counts. We are particularly anxious to extend our methods to
directed graphs where each node is indexed by (at least) two
characteristics such as in and out degree. Some investigations are
currently underway for the directed preferential attachment model
considered in 
\cite{
krapivsky:redner:2001,
bollobas:borgs:chayes:riordan:2003,
resnick:samorodnitsky:towsley:davis:willis:wan:2016,
resnick:samorodnitsky:2015}.
We also have a program evaluating various inferential methods for
estimating model parameters which requires asymptotic normality
results such as presented here.

\bibliography{/Users/sidresnick/Documents/SidFiles/bibfile}


\end{document}